\documentclass[11pt,a4paper]{amsart}

\usepackage{latexsym}
\usepackage{amsmath}
\usepackage{amsthm}
\usepackage{latexsym}
\usepackage{amssymb}
\usepackage{romannum}
\usepackage[a4paper , lmargin = {3cm} , rmargin = {3cm} , tmargin = {2.5cm} , bmargin = {2.5cm} ]{geometry}

\usepackage{hyperref}
\usepackage{tikz}

  	\newcommand{\Z}{\ensuremath{\mathbb{Z}}}
  	\newcommand{\R}{\ensuremath{\mathbb{R}}}   

   	\def\Aut{{\rm{Aut}}$(W_n)$}
   	\def\Spe{{\rm{Spe}}$(W_n)$}
    \def\FA{{\rm{F}}$\mathcal{A}$}

	\def\CAT{{\rm{CAT}}$(0)$}
	
	\def\F{{\rm{F}}}

\theoremstyle{plain}
	
\newtheorem*{NewTheoremA}{Theorem A}
\newtheorem*{NewCorollaryB}{Corollary B}
\newtheorem*{NewTheoremC}{Theorem C}
\newtheorem*{NewTheoremD}{Theorem D}
\newtheorem*{NewCorollaryE}{Corollary E}
\newtheorem*{NewCorollaryF}{Corollary F}

\newtheorem*{FarbF}{Farb's Fixed Point Criterion}

\newtheorem{theorem}{Theorem}[section]
\newtheorem{lemma}[theorem]{Lemma}

\newtheorem{corollary}[theorem]{Corollary}

\newtheorem{proposition}[theorem]{Proposition}
\newtheorem{definition}[theorem]{Definition}

\title{The automorphism group of the universal Coxeter group}
\author{Olga Varghese}
\thanks{Funded by the Deutsche 
Forschungsgemeinschaft (DFG, German Research Foundation) under Germany's 
Excellence Strategy EXC 2044-39068558, Mathematics M\"unster: Dynamics-Geometry-Structure}
\address{Olga Varghese\\
Department of Mathematics\\
M\"unster University\\ 
Einsteinstra\ss e 62\\
48149 M\"unster (Germany)}
\email{olga.varghese@uni-muenster.de}

\begin{document}

\pagenumbering{arabic}
\begin{abstract}
We study fixed point properties of the automorphism group of the universal Coxeter group \Aut. In particular, we prove that whenever \Aut\
acts by isometries on complete $d$-dimensional \CAT\ space with $d<\lfloor\frac{n}{2}\rfloor$, then it must fix a point. We also prove that \Aut\ does not have Kazhdan's property (T). Further, strong restrictions are obtained on homomorphisms of \Aut\ to groups that do not contain  a copy of ${\rm Sym}(n)$. 
\end{abstract}
\maketitle

\section{Introduction}
This article belongs to geometric group theory, a young research field which lies in the intersection of algebra, geometry and topology. Geometric group theory studies the interplay between algebraic and geometric properties of groups. In this article we explore the structure of the automorphism group of the universal Coxeter group of rank $n$, ${\rm Aut}(W_n):={\rm Aut}(\Z_2*\ldots*\Z_2)$, from geometric perspective. For small $n$ we have: ${\rm Aut}(W_1)\cong\left\{{\rm id}\right\}$ and ${\rm Aut}(W_2)\cong W_2$. 
It was proven by M\"uhlherr that for $n\geq 3$ there exists an injective homomorphism $\iota:{\rm Aut}(W_n)\hookrightarrow{\rm Aut}(F_{n-1})$, where $F_{n-1}$ is a free group of rank $n-1$, see  \cite{Muehlherr}. Further, the abelianization map $F_{n-1}\twoheadrightarrow \Z^{n-1}$ gives a natural epimorphism 
${\rm Aut}(F_{n-1})\twoheadrightarrow{\rm GL}_{n-1}(\Z)$. We obtain the following interaction of groups: 
$${\rm Aut}(W_n)\hookrightarrow{\rm Aut}(F_{n-1})\twoheadrightarrow{\rm GL}_{n-1}(\Z).$$
As a special case, we have:
${\rm Out}(W_3)\cong{\rm Out}(F_2)\cong{\rm GL}_2(\Z).$
Many algebraic and geometric properties are known for the groups ${\rm Aut}(F_n)$ and ${\rm GL}_{n-1}(\Z)$, but the research on the structure of the automorphism group of the universal Coxeter group is quite new.

We start to present results concerning fixed point properties of \Aut. A group $G$ is said to have property $\F\mathcal{A}_d$ if any action of $G$
by isometries on complete \CAT\ space of covering dimension $d$ 
has a fixed point. We prove

\begin{NewTheoremA}
If $n\geq 4$ and $d<\left\lfloor\frac{n}{2}\right\rfloor$, then \Aut\ has property $\F\mathcal{A}_{d}$. 
\end{NewTheoremA}
Concerning ${\rm Aut}(F_{n-1})$ and ${\rm GL}_{n-1}(\Z)$ similar results were proved in \cite{Varghese}. Recall that a group $G$ is said to satisfy Serre's property ${\rm F}\mathcal{A}$ if every action, without inversions, of $G$ on a simplicial tree has a fixed point. Serre proved that ${\rm GL}_n(\Z)$ has property $\F\mathcal{A}$ for $n\geq 3$. Regarding ${\rm Aut}(F_n)$ Bogopolski was the first to prove that this group also has property $\F\mathcal{A}$, see \cite{BogopolskiFA}. The next corollary follows immediately from Theorem A.
\begin{NewCorollaryB}
The group \Aut\ has Serre's property \FA\ if and only if  $n=1$ or $n\geq 4$. In particular, \Aut\ is isomorphic to a non-trivial amalgam if and only if $n=2$ or $n=3$.
\end{NewCorollaryB}

If a group has property ${\rm F}\mathcal{A}$ it is interesting to know if this group also has Kazhdan's property (T). For the group \Aut\ we prove
\begin{NewTheoremC}(see Theorem 3.5)
For $n\geq 2$, \Aut\ does not have Kazhdan's property (T).
\end{NewTheoremC}

We say that a finitely generated group $G$ is a \CAT\ group if $G$ acts properly, cocompactly by isometries on a \CAT\ metric space. It is known that ${\rm Aut}(W_2)$ and ${\rm Aut}(W_3)\cong{\rm Aut}(F_2)$ are \CAT\ groups, see Lemma 2.3 and \cite{Piggott}. For $n\geq 4$, we conjecture that ${\rm Aut}(W_n)$ is a \CAT\ group.

Our results concerning linear and free representations of \Aut\ rely on the following fact.
\begin{NewTheoremD}
For $n\geq 4$, let $\phi:{\rm Aut}(W_n)\rightarrow G$ be a group homomorphism. If $G$ does not contain an isomorphic image of ${\rm Sym}(n)$, then the image of $\phi$  is finite. In particular, if $n\geq 5$ then the image has cardinality at most $4$.
\end{NewTheoremD}

By Theorem D follows that the homomorphism
${\rm Aut}(W_n)\hookrightarrow{\rm Aut}(F_{n-1})\twoheadrightarrow{\rm GL}_{n-1}(\Z)\hookrightarrow{\rm GL}_{n-1}(\R)$
is minimal in the following
sense.
\begin{NewCorollaryE}
For $n\geq 4$, let $\rho:{\rm Aut}(W_n)\rightarrow{\rm GL}_d(K)$ be a linear representation over a field $K$. If $d<n-1$ and ${\rm char}(K)=0$ or ${\rm char}(K)\nmid n$, then the image of $\rho$ is finite.
\end{NewCorollaryE}
From representation theory of ${\rm Sym}(n)$ \cite[Chap.19, \S 8 Thm. 22]{Berkovich} we obtain that ${\rm GL}_d(K)$ does not contain an isomorphic image of ${\rm Sym}(n)$, hence by Theorem D the image of $\rho$ is finite.\\

We have ${\rm Aut}(W_n)\hookrightarrow{\rm Aut}(F_{n-1})$ and ${\rm Sym}(n+1)\nsubseteq{\rm Aut}(F_{n-1})$, see proof of Theorem A in \cite{BridsonVogtmann}. Further, since for $d<n-1$ the group ${\rm Sym}(n)\nsubseteq{\rm Aut}(F_d),{\rm Out}(F_d)$ and ${\rm GL}_d(\Z)$, see \cite[Lemma 2]{BridsonVogtmann}, we obtain the following results.

\begin{NewCorollaryF}
Let $n\geq 4$.
\begin{enumerate}
\item[(i)] If $d<n$, then every homomorphism $\phi:{\rm Aut}(W_n)\rightarrow{\rm Aut}(W_d)$ has finite image.
\item[(ii)] If $d<n-1$, then every homomorphism ${\rm Aut}(W_n)\rightarrow {\rm Aut}(F_d)$ has finite image. The same result is true for homomorphisms  ${\rm Aut}(W_n)\rightarrow{\rm Out}(F_d)$ and ${\rm Aut}(W_n)\rightarrow{\rm GL}_d(\Z)$.
\end{enumerate}
\end{NewCorollaryF}

\subsection*{Remarks} My research on the automorphism group of the universal Coxeter group was motivated by discussion with N. Leder about property $\F\mathcal{A}$ for automorphism groups of graph products. In his preprint \cite{Leder} he proved several results concerning Serre's property ${\rm F}\mathcal{A}$ for automorphism groups of free products of cyclic groups. Further, J. Flechsig wrote his Master's Thesis about algebraic and geometric properties of \Aut\ based on this article, see \cite{Flechsig}.

\section*{Acknowledgment}
The author thanks the referee for careful reading of the manuscript and for helpful remarks concerning \CAT\ groups.

\section{Preliminaries}

Let $(W_n, S_n)$ be the universal Coxeter system of rank $n$, i.e $S_n$ is a set $\left\{s_1,\ldots, s_n\right\}$ and the group $W_n$ is given by the following presentation $W_n=\langle S_n\mid s^2_1,\ldots, s^2_n\rangle$.
By \Aut\ we denote the automorphism group of the group $W_n$.

The purpose of this section is to give a generating set of the group \Aut\ for $n\geq 2$. 
Let us first introduce a notation for some elements of \Aut.   We define the partial conjugations $\sigma_{ij}$ and  permutations $\alpha_\pi$ for $1\leq i\neq j\leq n$ and $\pi\in{\rm Sym}(n)$ as follows:
\[
\begin{matrix}
\sigma_{ij}(s_{k}):=\begin{cases} s_{i}s_{j}s_{i} & \mbox{if $k=j$,}  \\ 
				     s_{k} & \mbox{if $k\neq j$,}   
\end{cases}
& \text{$ $}
\alpha_\pi(s_{k}):=s_{\pi(k)}
\end{matrix}
\]
It was proven by M\"uhlherr in \cite[Theorem B]{Muehlherr}) that for~$n\geq 2$ the group \Aut\  is generated by the set 
\[
\left\{\alpha_\pi, \ \sigma_{ij}\mid \pi\in{\rm Sym}(n), 1\leq i\neq j\leq n \right\}.
\]
The subgroup generated by $\left\{\alpha_\pi\mid\pi\in{\rm Sym}(n)\right\}$ is isomorphic to ${\rm Sym}(n)$. It is well known that this group is generated by the involutions $(i,i+1)$ with $i=1, \ldots, n-1$. Hence the group \Aut\ is generated by the set 
\[
\left\{\alpha_{(i, i+1)}, \ \sigma_{kl}\mid 1\leq i\leq n-1, 1\leq k\neq l\leq n \right\}.
\]
Further we have the following relations: $\alpha_\pi\sigma_{ij}\alpha^{-1}_\pi=\sigma_{\pi(i)\pi(j)}$. 
Therefore we obtain the following generating set for \Aut:
\begin{proposition}
\label{GenAut}
Let $n\geq 2$ and $1\leq k\neq l\leq n$. The group \Aut\ is generated by 
\[
Y:=\left\{ \sigma_{kl}, \ \alpha_{(i,i+1)} \mid i=1,\ldots, n-1\right\}.
\]
\end{proposition}

\begin{corollary}
\label{epi}
For $n\geq 2$, the group \Aut\ is a quotient of a Coxeter group $G$ whose Coxeter graph looks like as follows

\begin{figure}[h]
\begin{center}
\begin{tikzpicture}
 \draw (0,0.1)--(0.9,1);
 \draw (0,0.1)--(0,1.9);
 \draw (0,1.9)--(0.9,1);
 \draw (1.1,1)--(2.4,1);
 \draw (2.6,1)--(3.9,1);
 \draw (4.1,1)--(5.4,1);
 \draw (5.6,1)--(6.9,1);
 \draw (8.6,1)--(9.9,1);
 \draw (0,0) circle (3pt);
 \draw (0,-0.5) node {\footnotesize{$\overline{\sigma_{12}}$}};
 \draw (0,2) circle (3pt);
 \draw (0,2.5) node {\footnotesize{$(1, 2)$}};
 \draw (1,1) circle (3pt);
 \draw (1.25,0.5) node {\footnotesize{$(2, 3)$}};
 \draw (2.5,1) circle (3pt);
 \draw (2.5,1.5) node {\footnotesize{$(3, 4)$}};
 \draw (4,1) circle (3pt);
 \draw (4,0.5) node {\footnotesize{$(4, 5)$}};
 \draw (5.5,1) circle (3pt);
 \draw (5.5,1.5) node {\footnotesize{$(5, 6)$}};
 \draw (7.80,1) node {$\ldots$};
 \draw (10,1) circle (3pt);
 \draw (10,0.5) node {\footnotesize{$(n-1, n)$}};
 
 \draw (0.5,0.35) node {\footnotesize{$4$}};
 
  \draw (-0.3,0.85) node {\footnotesize{$\infty$}};

\end{tikzpicture}
\caption{}
\end{center}
\end{figure}

Further, the subgroup of \Aut\ generated by $\left\{\sigma_{12}, \alpha_{(i,i+1)}\mid i=2, \ldots, n-1\right\}$ is finite.
\end{corollary}
\begin{proof}
The elements in $Y$ are involutions. We define a homomorphism $f:G\rightarrow{\rm Aut}(W_n)$
as follows:
for $1\leq i\leq n-1$: $(i,i+1)\mapsto \alpha_{(i,i+1)}$ and $\overline{\sigma_{12}}\mapsto\sigma_{12}$. This map is well-defined and surjective. The subgroup of $G$ generated by $\left\{\overline{\sigma_{12}}, (i,i+1)\mid i=2,\ldots, n-1\right\}$ is a Coxeter group of type $B_{n-1}$ and is therefore finite. The subgroup of \Aut\ generated by 
$\left\{\sigma_{12}, \alpha_{(i,i+1)}\mid i=2, \ldots, n-1\right\}$ is the image of the subgroup generated by $\left\{\overline{\sigma_{12}}, (i,i+1)\mid i=2,\ldots, n-1\right\}$ under $f$ and therefore also finite.
\end{proof}

Let us consider the map $\epsilon:W_n\rightarrow\left\{-1,1\right\}$  which sends each generator $s_i$ to $-1$. M\"uhlherr proved in \cite{Muehlherr} that the kernel of $\epsilon$ is a characteristic subgroup of $W_n$ and it is isomorphic to $F_{n-1}$. Further the set $\left\{x_i:=s_is_{i+1}\mid i=1, \ldots, n-1\right\}$ is a basis of this kernel and the natural map
\[
\iota:{\rm Aut}(W_n)\rightarrow{\rm Aut}({\rm ker}(\epsilon))\cong{\rm Aut}(F_{n-1})
\]
is for $n\geq 3$ a monomorphism. An easy calculation yields:
\begin{lemma}
\label{Iso}
The map $\iota:{\rm Aut}(W_3)\rightarrow{\rm Aut}(F_2)$ is an isomorphism.
\end{lemma}
For a generating set of ${\rm Aut}(F_2)$ see \cite[Proposition 1]{Kramer}.

\section{Fixed point properties of \Aut}
Definitions and properties concerning \CAT\ spaces can be found in \cite{BH}.
We need the following crucial definition.

\begin{definition}
\begin{enumerate}
\item[(i)] A  group
$G$ has Serre's property $\F\mathcal{A}$ if  any simplicial action without inversions on a simplicial tree has a fixed point.
\item[(ii)] A  group
$G$ has property $\F\mathcal{A}_d$ if  any  isometric action  on a complete \CAT\ space of covering dimension $d$
has a fixed point.
\end{enumerate}
\end{definition}

Our main technique in the proof of Theorem A is based on the following criterion, see \cite{Varghese} for the proof.  
\begin{FarbF}
\label{HellyGroup}
Let $G$ be a group, $Y$ a finite generating set of~$G$ and~$X$ a complete $d$-dimensional \CAT\ space. If
$\Phi:G\rightarrow{\rm Isom}(X)$ is a homomorphism such that each $(d+1)$-element subset of $Y$ has a fixed point in $X$, then $G$ has a fixed point in $X$.
\end{FarbF}

The following version of the Bruhat-Tits Fixed Point Theorem \cite[\Romannum{2} 2.8]{BH} is crucial for our arguments. 
\begin{proposition}
\label{boundedOrbit}
 Let $G$ be a group acting on a complete \CAT\ space $X$ by isometries. Then the following conditions are equivalent: 
\begin{enumerate}
 \item[$(i)$] The group $G$ has a global fixed point.
 \item[$(ii)$] Each orbit of $G$ is bounded.
 \item[$(iii)$] The group $G$ has a bounded orbit.
\end{enumerate}
If the group $G$ satisfies one of the conditions above, then $G$ is called bounded on $X$.\\
\end{proposition}
The implications $i)\Rightarrow ii)$ and $ii)\Rightarrow iii)$ are trivial, and $iii)\Rightarrow i)$ is proven in \cite[\Romannum{2} 2.8]{BH}.
$ $\\

The following corollary  is an easy consequence of Proposition \ref{boundedOrbit}.

\begin{corollary}
\label{comm}
Let $G_{1}$, $G_{2}$ be groups, $X$ a complete \CAT\ space and 
\begin{align*}
\phi_{1}&:G_{1}\rightarrow{\rm Isom}(X),\\
\phi_{2}&:G_{2}\rightarrow{\rm Isom}(X) 
\end{align*}
 homomorphisms. If $G_{1}$ and $G_{2}$ are bounded on $X$ and $\phi_{1}(g_{1})\circ\phi_{2}(g_{2})=\phi_{2}(g_{2})\circ\phi_{1}(g_{1})$ for all $g_{1}$ in $G_{1}$ and $g_{2}$ in $G_{2}$, then the map
\begin{align*}
 \phi_{1}\times\phi_{2}:G_{1}\times G_{2}&\rightarrow{\rm Isom}(X)\\
                           (g_{1}, g_{2})&\mapsto\phi(g_{1})\circ\phi(g_{2})
\end{align*}
is a homomorphism and $G_{1}\times G_{2}$ is bounded on $X$.
\end{corollary}
For the proof of Theorem A we need one more ingredient.
\begin{proposition}(\cite[3.6]{MappingClassFA})
\label{conjugates}
 Let $k$ and $l$ be in $\mathbb{N}_{>0}$ and let $X$ be a complete $d$-dimensional \CAT\ space, with $d<k\cdot l$. Let $S$ be a subset of ${\rm Isom}(X)$ and let $S_{1},\ldots, S_{l}$ be conjugates of $S$ such that $[S_{i}, S_{j}]=1$ for $i\neq j$. If each $k$-element subset of $S$ has a fixed point in $X$, then each finite subset of $S$ has a fixed point in $X$.
\end{proposition}

Now we are ready to prove Theorem A. 

\begin{NewTheoremA}
\label{MainTheorem}
If $n\geq 4$ and $d<\left\lfloor\frac{n}{2}\right\rfloor$, then \Aut\ has property $\F\mathcal{A}_{d}$. 
\end{NewTheoremA}
\begin{proof}
Let $X$ be a $d$-dimensional complete \CAT\ space with $d<\left\lfloor\frac{n}{2}\right\rfloor$ and 
\[
 \Phi:{\rm Aut}(W_n)\rightarrow{\rm Isom}(X)
\]
an action of \Aut\ on $X$. By Proposition \ref{GenAut} the group \Aut\ is generated by the set 
\[
Y:=\left\{ \alpha_{(i,i+1)}, \sigma_{12}\mid i=1,\ldots, n-1\right\}.
\]
By Corollary \ref{epi} we have an epimorphism
$f: G\twoheadrightarrow{\rm Aut}(W_n)$. Now we consider the following homomorphism
\[
\Phi\circ f: G\twoheadrightarrow{\rm Aut}(W_n)\rightarrow{\rm Isom}(X).
\] 

It is obvious that if a subgroup of $G$ has a fixed point in $X$, then the image of this subgroup under $f$, a subgroup in \Aut, also has a fixed point.

We know by Proposition \ref{boundedOrbit} that each $1$-element subset of $Y$ has a fixed point. Now we assume that each $k$-element subset with $k\leq d$ of $Y$ has a fixed point. Let $Y'$ be a $(k+1)$-element subset of $Y$.

If $\sigma_{12}$ is not in $Y'$, then $\langle Y'\rangle$ is a finite subgroup of \Aut\ and this subgroup has by Proposition \ref{boundedOrbit} a fixed point. 

If $\sigma_{12}$ is in $Y'$, we consider the corresponding Coxeter diagram of $\langle Y'\rangle\subseteq G$. If it is not connected, then it follows from the hypothesis and from Corollary \ref{comm} that $\langle Y'\rangle $ has a fixed point. If the Coxeter diagram of $\langle Y'\rangle\subseteq G$ is connected, then we have the following cases:
\begin{enumerate}
 \item $Y'=\left\{\sigma_{12},\alpha_{(1,2)}, \alpha_{(2,3)}, \alpha_{(3,4)},\ldots,\alpha_{(k,k+1)}\right\}$,
 \item $Y'=\left\{\sigma_{12}, \alpha_{(2,3)}, \alpha_{(3,4)}, \alpha_{(5,6)},\ldots,\alpha_{(k+1,k+2)}\right\}$.
\end{enumerate}

If $Y'$ is equal to $\left\{\sigma_{12},\alpha_{(1,2)}, \alpha_{(2,3)}, \alpha_{(3,4)},\ldots,\alpha_{(k,k+1)}\right\}$, then  we define the permutations 
\[
 \tau_{i}:=(1, (k+1)\cdot(i-1)+1)(2, (k+1)\cdot(i-1)+2)\ldots(k+1,(k+1)\cdot(i-1)+k+1)
\]
and the sets 
\[
 S_{i}:=\alpha_{\tau_{i}}Y'\alpha^{-1}_{\tau_{i}}
\]
for $i\in\left\{1,\ldots,\left\lfloor\frac{n}{k+1}\right\rfloor\right\}$. The sets $S_{1},\ldots, S_{\left\lfloor\frac{n}{k+1}\right\rfloor}$ have the property that $[S_{i},S_{j}]=1$ for~$i\neq j$. By the assumption each $k$-element subset of $Y'$ has a fixed point and it follows from Proposition \ref{conjugates} that for $d<k\left\lfloor\frac{n}{k+1}\right\rfloor$ the set $Y'$ has a fixed point. An easy calculation shows, that $\left\lfloor\frac{n}{2}\right\rfloor\leq k\left\lfloor\frac{n}{k+1}\right\rfloor$ for $1\leq k\leq d$.

If $Y'$ is equal to $\left\{\rho_{12}, (x_{2},x_{3}), (x_{3},x_{4}),\ldots,(x_{k+1},x_{k+2})\right\}$, then $\langle Y'\rangle$ is finite (Coxeter group of type $B_{k+1}$).

By Farb's Fixed Point Criterion it follows that \Aut\ has a global fixed point.
\end{proof}

\begin{NewCorollaryB}
The group \Aut\ has Serre's property \FA\ iff $n=1$ or $n\geq 4$. In particular, \Aut\ is isomorphic to a non-trivial amalgam iff $n=2, 3$.
\end{NewCorollaryB}
\begin{proof}
For $n=1$ we have ${\rm Aut}(W_1)\cong\left\{{\rm id}\right\}$ and hence this group has property \FA. For $n=2$ the group $W_2$ is the infinite dihedral group and by \cite[1.4]{Thomas} ${\rm Aut}(W_2)\cong W_2$ and therefore this group does not have property \FA, see \cite[6.1 Theorem 15]{Serre}. By Lemma \ref{Iso} we have ${\rm Aut}(W_3)\cong{\rm Aut}(F_2)$ and this group does not have property \FA, see \cite{BogopolskiFA}. By Theorem A the group \Aut\ has property \FA\ for $n\geq 4$. 

By Proposition \ref{GenAut} the group \Aut\ is generated by elements of finite order and therefore does not surjects to $\mathbb{Z}$. Thus, \Aut\ is isomorphic to a non-trivial amalgam iff $n=2, 3$ by \cite[\S 6 Theorem 15]{Serre}.
\end{proof}

The next question we want to investigate is the following: {\it do subgroups of finite index in \Aut\ have property \FA?}

Let $G$ be any group. Then the group of its automorphisms acts on the conjugacy classes of involutions of $G$. The kernel of this action is called the group of special automorphisms and denoted by ${\rm Spe}(G)$, it contains the group of inner automorphisms, denoted by ${\rm Inn}(G)$. For the universal Coxeter group $W_n$ there are $n$ conjugacy classes of involutions, and we have ${\rm Aut}(W_n)/{\rm Spe}(W_n)\cong {\rm Sym}(n)$. 
\begin{theorem}
For $n\geq 2$, the group \Spe\ doesn't have property \FA. Hence \Aut\ and \Spe\ don't have Kazhdan's property (T).
\end{theorem}
\begin{proof}
Let $\pi: W_n\twoheadrightarrow W_2$ be the projection by sending $s_1\mapsto s_1, s_2\mapsto s_2$ and $s_i\mapsto 1$ for $i\geq 3$. The kernel of $\pi$ is characteristic under ${\rm Spe}(W_n)$. Therefore we obtain the following map 
$\Psi:{\rm Spe}(W_n)\rightarrow{\rm Spe}(W_2)$ which is surjective. Since  ${\rm Spe}(W_2)=\langle \sigma_{12}, \sigma_{21}\rangle\cong W_2$, see \cite{Muehlherr2}, this group doesn't have property \FA, therefore \Spe\ also doesn't have property \FA. 
By the result of Watatani \cite{Watatani} follows that \Spe\ doesn't have property (T). Since property (T) descends to finite index subgroups it follows that \Aut\ doesn't have property (T) either.
\end{proof}

If a group does not have property (T) it is natural to ask if this group is amenable. Concerning the group \Aut\ we prove:
\begin{proposition}
The group \Aut\ is amenable iff $n=1,2$.
\end{proposition}
\begin{proof}
For $n=1, 2$ we have: ${\rm Aut}(W_1)=\left\{{\rm id}\right\}$, ${\rm Aut}(W_2)\cong W_2\cong\Z\rtimes\Z_2$. By \cite[G.2.1]{Bekka} $\Z$  and $\Z_2$ are amenable and by \cite[G.3.6]{Bekka} follows that $\Z\rtimes\Z_2$ is amenable.

For $n\geq 3$ let us consider the monomorphism $\iota:{\rm Aut}(W_n)\rightarrow{\rm Aut}(F_{n-1})$ again.
We denote by $g_{x_i}\in{\rm Inn}(F_{n-1})$ the conjugation with $x_i$ for $i=1, \ldots, n-1$. 
We have:
\[
\iota(\sigma_{12}\sigma_{13}\ldots\sigma_{1n}\sigma_{21}\sigma_{23}\ldots\sigma_{2n})=g_{x_1}
\]
and
\[
\iota(\sigma_{21}\sigma_{23}\ldots\sigma_{2n}\sigma_{31}\sigma_{32}\ldots\sigma_{3n})=g_{x_2}
\] 
We obtain:
\[
\langle \sigma_{12}\sigma_{13}\ldots\sigma_{1n}\sigma_{21}\sigma_{23}\ldots\sigma_{2n}, \ \sigma_{21}\sigma_{23}\ldots\sigma_{2n}\sigma_{31}\sigma_{32}\ldots\sigma_{3n}\rangle\cong\langle g_{x_1}, g_{x_2}\rangle \cong F_2.
\]
Hence $F_{2}$ is a subgroup of \Aut. By \cite[G.3.5]{Bekka} follows that \Aut\ is not amenable.
\end{proof}

\section{Proof of Theorem D}

\begin{proof}
We denote by $\Sigma_n=\left\{\alpha_{\pi}\mid \pi\in{\rm Sym}(n)\right\}\subseteq{\rm Aut}(W_n)$ and by $A_n=\left\{\alpha_\pi\mid \pi\in{\rm Alt}(n)\right\}\subseteq \Sigma_n$. We have $\Sigma_n\cong{\rm Sym}(n)$ and $A_n\cong {\rm Alt}(n)$. Let $K$ be the kernel of $\phi_{|\Sigma_n}$. Since $n\geq 4$ and $K\neq 1$, we must have $A_n\subseteq K$ or $K=\left\{id, \alpha_{(12)(34)}, \alpha_{(13)(24)}, \alpha_{(14)(23)}\right\}$.
If $A_n\subseteq K$, then we consider the following generating set of \Aut:
\[
\left\{\alpha_\pi, \alpha_{(1,2)}, \sigma_{34}\mid \pi\in {\rm Alt}(n) \right\}.
\]
Hence the image of $\phi$ is generated by
\[
\left\{\phi(\alpha_\pi), \phi(\alpha_{(1,2)}), \phi(\sigma_{34})\mid \pi\in{\rm Alt}(n)\right\}.
\] 
Since $A_n\subseteq K$, we have
\[
\phi({\rm Aut}(W_n))=\langle \phi(\alpha_{(1,2)}), \phi(\sigma_{34})\rangle.
\]
The elements $\alpha_{(1,2)}$ and $\sigma_{34}$ are commuting involutions, therefore $\phi({\rm Aut}(W_n))\subseteq\Z_2\times\Z_2$.

If $K=\left\{id, \alpha_{(12)(34)}, \alpha_{(13)(24)}, \alpha_{(14)(23)}\right\}$, then $n=4$ and we consider the following generating set of ${\rm Aut}(W_4)$:
\[
\left\{K, \alpha_{(2,3)}, \alpha_{(3,4)}, \sigma_{12}\right\}.
\]
The image of $\phi$ is generated by
\[
\left\{\phi(\alpha_{(2,3)}), \phi(\alpha_{(3,4)}), \phi(\sigma_{12})\right\}.
\]
By Corollary \ref{epi} the subgroup of ${\rm Aut}(W_4)$ generated by 
$\left\{\alpha_{(2,3)}, \alpha_{(3,4)}, \sigma_{12}\right\}$ is finite. Hence the image of $\phi$ is also finite. 
\end{proof}

\end{document}